\documentclass[smallextended]{svjour3}

\usepackage{amssymb,graphicx,amsmath}
\usepackage{color,fullpage,hyperref,xcolor}

\usepackage{graphicx,pgf,tikz,epstopdf}
\usetikzlibrary{fadings,shapes,shapes.misc}

\usepackage{amsopn}

\usepackage[boxed,commentsnumbered]{algorithm2e}
\usepackage{rotating}

\usepackage{algpseudocode}

\makeatletter
\newcommand\footnoteref[1]{\protected@xdef\@thefnmark{\ref{#1}}\@footnotemark}
\makeatother

\newcommand{\real}{I\!\!R}

\newcommand{\TheTitle}{A Penalty Method for Rank Minimization Problems
in Symmetric Matrices}

\title{{\TheTitle}\thanks{This work was supported in part by the Air Force Office of Sponsored Research under grants FA9550-08-1-0081 and FA9550-11-1-0260
and by the National Science Foundation
			under Grants Number CMMI-1334327 and DMS-1736326.}}

\author{
  Xin Shen\thanks{Monsanto, St.\ Louis, MO.}
  \and
  John E. Mitchell\thanks{Department of Mathematical Sciences,
    Rensselaer Polytechnic Institute,
    Troy,  NY 12180
     (\href{mailto:mitchj@rpi.edu}{mitchj@rpi.edu}, \url{http://www.rpi.edu/\~mitchj}).}
}

\usepackage{amsopn}

\begin{document}

\maketitle

\begin{abstract}
The problem of minimizing the rank of a symmetric
positive semidefinite matrix subject to
constraints can be cast equivalently as a semidefinite program
with complementarity constraints (SDCMPCC). The formulation
requires two positive semidefinite matrices to be complementary.
This is a continuous and nonconvex reformulation of the rank
minimization problem.
We investigate calmness of locally optimal solutions to the
SDCMPCC formulation and hence show that any locally optimal
solution is a KKT point.

We develop a penalty formulation of the problem. We present
calmness results for locally optimal solutions to the penalty
formulation. We also develop a proximal alternating linearized
minimization (PALM) scheme for the penalty formulation, and
investigate the incorporation of a momentum term into the
algorithm. Computational results are presented.

\noindent
Keywords:
Rank minimization,
penalty methods,
alternating minimization

\noindent
AMS Classification:
  90C33,  %
  90C53  %

\end{abstract}

\section{Introduction to the Rank Minimization Problem}

Recently rank constrained optimization problems have received increasing interest because of their wide application in many fields including statistics, communication and signal processing \cite{fazel2001rank,srebro2003weighted}.
In this paper we mainly consider one genre of the problems whose objective is to minimize the rank of a matrix subject to a given set of constraints.
We consider the slightly more general form below:
\begin{equation} \label{RankGeneral}
\begin{array}{ll}
\displaystyle{
{\operatornamewithlimits{\mbox{minimize}}_{X \, \in \, \mathbb{R}^{m\times n}}}
} & \mbox{rank}(X) \, + \, \psi(X) \\ [0.15in]
\mbox{subject to} & X \in \mathcal{C}
\end{array}
\end{equation}
where $\mathbb{R}^{m\times n}$ is the space of size m by n matrices,
$\psi(X)$ is a Lipschitz continuous function,
and $\mathcal{C}$ is the feasible region for X;
$\mathcal{C}$ is not necessarily convex.

The class of problems has been considered  computationally challenging because of its nonconvex nature. The rank function is also highly discontinuous, which makes rank minimization problems hard to solve.
Methods using nonconvex optimization to solve rank minimization problems
include~\cite{kmOh2010,llsWei2016,sLuo2015,tWei2016}.
In contrast to the method in this paper, these references work with an explicit low
rank factorization of the matrix of interest.
Other methods based on a low-rank factorization include the thresholding methods
\cite{candes1,cambierabsil2016,vandereycken2013,weiCCL2016b}.
Our approach works with a nonconvex nonlinear optimization problem that is
an exact reformulation of the rank minimization problem.

The exact reformulation of the rank minimization problem is a mathematical program with
semidefinite cone complementarity constraints (SDCMPCC).
Similar to the LPCC formulation for the $\ell_0$ minimization
problem~\cite{bkSchwartz2014,FMPSWachter13},
the advantage of the SDCMPCC formulation is that it can be expressed as a smooth nonlinear program,
thus it can be solved by general nonlinear programming algorithms.
The purpose of this paper is to investigate whether nonlinear semidefinite programming algorithms can be applied to solve the SDCMPCC formulation and examine the quality of solution returned by the nonlinear algorithms. We're faced with two challenges. The first one is the nonconvexity of the SDCMPCC formulation, which means that we can only assure that the solutions we find are locally optimal. The second is that most nonlinear algorithms use KKT conditions as their termination criteria. 
Since a general SDCMPCC formulation is not well-posed because of the complementarity constraints, i.e, KKT stationarity may not hold at local optima, 
there might be some difficulties with the convergence of these algorithms.
We show in Theorem~\ref{thm:KKTsdcmpcc} that any locally optimal point for the
SDCMPCC formulation of the
rank minimization problem does indeed satisfy the KKT conditions.

When $\psi(X)\equiv 0$, a popular approach to choosing~$X$ is to use the
nuclear norm approximation
\cite{fazel2001rank,vandenberghe6,MaGoldfarbChenMP2010,candes1},
a convex approximation of the original rank minimization problem.
The nuclear norm of a matrix $X\in\mathcal{R}^{m\times n}$ is defined as the sum of its singular values:
\begin{displaymath}
||X||_*\,=\,\sum \sigma_i\,=\,trace(\sqrt{X^T X})
\end{displaymath}
In the approximated problem, the objective is to find a matrix with the minimal nuclear norm
\begin{equation} \label{RankGeneral_nuclear}
\begin{array}{ll}
\displaystyle{
{\operatornamewithlimits{\mbox{minimize}}_{X \, \in \, \mathbb{R}^{m\times n}}}
} &||X||_*\\ [0.15in]
\mbox{subject to} & X \in \mathcal{C}
\end{array}
\end{equation}
The nuclear norm is convex and continuous. Many algorithms have been developed previously to find the optimal solution to the nuclear norm minimization problem, including interior point methods \cite{vandenberghe6}, singular value thresholding \cite{candes1}, Augmented Lagrangian method \cite{lin2010augmented}, proximal gradient method \cite{liu2012implementable}, subspace selection method \cite{hsieh2014nuclear}, reweighting methods~\cite{mFazel2012},
and so on.
These methods have been shown to be efficient and robust in solving large scale nuclear norm minimization problems in some applications.
Previous works provided some explanation for the good performance for convex approximation by showing that nuclear norm minimization and rank minimization is equivalent under certain assumptions.
Recht et al.~\cite{recht2} presented a version of the restricted isometry property for a rank minimization problem.
Under such a property the solution to the original rank minimization problem can be exactly recovered by solving the nuclear norm minimization problem.  However, these properties are too strong and hard to validate, and the equivalence result cannot be extended to the general case.
Zhang et al.~\cite{zhang2013counterexample}
gave a counterexample in which the nuclear norm fails to find the matrix with the minimal rank.

In this paper, we focus on the case of symmetric matrices~$X$.
Let $\mathbb{S}^n$ denote the set of symmetric $n \times n$ matrices,
and $\mathbb{S}^n_+$ denote the cone of $n\times n$ symmetric positive semidefinite matrices.
The set $\mathcal{C}$ is taken to be the intersection of~$\mathbb{S}^n_+$ with another convex set,
taken to be an affine manifold in our computational testing.
Unless otherwise stated, the norms we use are the Euclidean 2-norm for vectors
and the Frobenius norm for matrices.

To improve the performance of the nuclear norm minimization scheme in the case of symmetric
positive semidefinite matrices,
a reweighted nuclear norm heuristic was put forward by Mohan et al.~\cite{mohan2010reweighted}.
In each iteration of the heuristic a reweighted nuclear norm minimization problem is solved, which takes the form:
\begin{equation} \label{Rankreweighted}
\begin{array}{ll}
\displaystyle{
{\operatornamewithlimits{\mbox{minimize}}_{X \, \in \, \mathbb{S}^n}}
} &\langle W, X\rangle \\ [0.15in]
\mbox{subject to} & X \in \mathcal{C} \cap \mathbb{S}^n_+
\end{array}
\end{equation}
where $W$ is a positive semidefinite matrix, with W based upon the result of the last iteration.
As with the standard nuclear norm minimization, the method only applies to problems with special structure.
The lack of theoretical guarantee for these convex approximations in general problems motivates us to turn to
the exact formulation of the rank function.
In our computational testing, we compare the results obtained with our approach to those obtained
through optimizing the nuclear norm.

We now summarize the contents of the paper.
Throughout, we work with the set of symmetric positive semidefinite matrices~$\mathbb{S}^n_+$.
The equivalent continuous reformulation of (\ref{RankGeneral}) is presented in
Section~\ref{sec:sdcmpcc}.
We show that any local minimizer for the continuous reformulation satisfies appropriate
KKT conditions in Section~\ref{sec:kkt} when $\mathcal{C}$ is given
by the intersection of $\mathbb{S}^n_+$ and the set of solutions
to a collection of continuous inequalities.
A penalty approach is described in Section~\ref{sec:penalty}
in the case when $\mathcal{C}$ is the intersection of $\mathbb{S}^n_+$
and an affine manifold.
An alternating approach to solve the penalty formulation is presented in
Section~\ref{sec:PALM}, with test results on Euclidean distance matrix completion problems
given in Section~\ref{sec:comp}.

\section{Semidefinite Cone Complementarity Formulation for Rank Minimization Problems
\label{sec:sdcmpcc}}

A mathematical program with semidefinite cone complementarity constraints (SDCMPCC) is a special case of a mathematical program with complementarity constraints (MPCC). In SDCMPCC problems the constraints include complementarity between matrices rather than vectors. 
The general SDCMPCC program takes the following form:
\begin{equation} \label{GSDCMPCC}
\begin{array}{ll}
\displaystyle{
{\operatornamewithlimits{\mbox{minimize}}_{x \, \in \, \mathbb{R}^q}}
} &f(x)\\ [0.15in]
\mbox{subject to} &  g(x) \,\leq\, 0\\
& h(x)\,=\,0\\
& \mathbb{S}^n_+ \,\ni\, G(x)\,\perp\,H(x) \,\in\,\mathbb{S}^n_+
\end{array}
\end{equation}
where $f:\real^q \rightarrow \real$, $h:\real^q\rightarrow\real^p$, $g:\real^q \rightarrow \real^m$,
$G:\real^q\rightarrow\mathbb{S}^n$ and $H:\real^q \rightarrow \mathbb{S}^n$.
The requirement $G(x)\,\perp\,H(x)$ for $G(x), H(x) \in \mathbb{S}^n_+$
is that the Frobenius inner product of $G(x)$ and $H(x)$ is equal to 0, where the Frobenius inner product of two matrices $A\in\mathbb{R}^{m\times n}$ and $B\in\mathbb{R}^{m\times n}$ is defined as
\begin{displaymath}
\langle A, B\rangle \,=\,trace(A^T B) .
\end{displaymath}
We define
\begin{equation}
c(x) \, := \, \langle G(x), H(x) \rangle .
\end{equation}
It is shown in Bai et al.~\cite{lijie2} that (\ref{GSDCMPCC}) can be reformulated
as a convex conic completely positive optimization problem.
However, the cone in the completely positive formulation does not have a
polynomial-time separation oracle.

An SDCMPCC can be written as a nonlinear semidefinite program.
Nonlinear semidefinite programming recently received much attention because of its wide applicability.
Yamashita and Yabe~\cite{yamashita2015survey} surveyed numerical methods for solving nonlinear SDP programs, including Augmented Lagrangian methods~\cite{kocvara2}, sequential SDP methods and primal-dual interior point methods.
However, there is still  much room for research in both theory and practice with solution methods, especially when the size of problem becomes large.

An SDCMPCC is a special case of a nonlinear SDP program. It is hard to solve in general.
In addition to the difficulties in general nonlinear semidefinite programming,
the complementarity constraints pose challenges to finding the local optimal solutions since the KKT condition may not hold at local optima.
Previous work showed that optimality conditions in MPCC, such as M-stationary, C-Stationary and Strong Stationary, can be generalized into the class of SDCMPCC.
Ding et al.~\cite{dsye2010} discussed various kinds of first order optimality conditions of an SDCMPCC
and their relationship with each other.

An exact reformulation of the rank minimization problem
using semidefinite cone constraints is due to Ding et al.~\cite{dsye2010}.
We work with a special case of (\ref{RankGeneral}), in which the matrix variable
$X\in \mathbb{R}^{n \times n}$ is restricted to be symmetric and positive semidefinite.  The special case takes the form:
\begin{equation} \label{RankPSD}
\begin{array}{ll}
\displaystyle{
{\operatornamewithlimits{\mbox{minimize}}_{X \, \in \, \mathbb{S}^n}}
} &\mbox{rank}(X) \, + \, \psi(X) \\ [0.15in]
\mbox{subject to} & X \in \tilde{\mathcal{C}}\\[0.15in]
& X\,\succeq \,0
\end{array}
\end{equation}
By introducing an auxilliary variable $U\in\mathbb{R}^{n\times n}$, we can model Problem (\ref{RankPSD}) as a mathematical program with semidefinite cone complementarity constraints:
\begin{equation} \label{CompleRankPSD}
\begin{array}{ll}
\displaystyle{
{\operatornamewithlimits{\mbox{minimize}}_{X,U \, \in \, \mathbb{S}^n}}
} &n\,-\,\langle I,\,U\rangle \, + \, \psi(X) \\ [0.15in]
\mbox{subject to} & X \in \tilde{\mathcal{C}}\\[0.15in]
& 0\,\preceq\, X\,\perp\,U\,\succeq\,0\\[0.15in]
& 0\,\preceq\,I\,-\,U\\[0.15in]
& X\,\succeq\, 0\\[0.15in]
& U\,\succeq\,0
\end{array}
\end{equation}
The equivalence between Problem (\ref{RankPSD}) and Problem (\ref{CompleRankPSD}) can be verified by a proper assignment of U for given feasible X. Suppose X has the eigenvalue decomposition:
\begin{equation}\label{XEig}
X\,=\,P^T \Sigma P
\end{equation}
Let $P_{0}$ be the matrix composed of columns in P corresponding to zero eigenvalues. We can set:
\begin{equation}\label{AssignU}
U\,=\,P_{0} \,P_{0}^T
\end{equation}
It is obvious that 
\begin{equation}
\mbox{rank}(X)\,=\,n\,-\,\langle I, U\rangle 
\end{equation}
It follows that:
\begin{displaymath}
Opt(\ref{RankPSD})\,\geq\,Opt(\ref{CompleRankPSD})
\end{displaymath}
The opposite direction of the above inequality can be easily validated by the complementarity constraints. If there exists any feasible matrix pair $(X,U)$ with the trace of U greater than $n - Opt(\ref{RankPSD})$, the complementarity constraints would be violated:
since all the eigenvalues of $U$ are no larger than 1, the rank of $U$ is at least
as large as its trace, so the rank of $X$ would be smaller than the optimal value of~(\ref{RankPSD}).

The complementarity formulation can be extended to cases where the matrix variable $X\in \mathbb{R}^{m\times n}$ is neither positive semidefinite nor symmetric. One way to deal with nonsymmetric $X$ is to introduce an auxilliary variable Z:
\begin{displaymath}
Z\,=\,\left[\begin{array}{clcr}
G&X^T\\
X&B\end{array}\right]\succeq 0
\end{displaymath} 
Liu et al.~\cite{vandenberghe6} has shown that for any matrix $X$, we can find matrix $G$ and $B$ such that $Z\succeq 0$ and $\mbox{rank}(Z)=\mbox{rank}(X)$. The objective is to minimize the rank of matrix $Z$ instead of $X$.

A drawback of the above extension is that it might introduce too many variables. An alternative way is to modify the complementarity constraint. If $m>n$, the rank of matrix $X$ must be bounded by n and equals the rank of matrix $X^T X\in \mathbb{S}^n_+$. $X^T X$ is both symmetric and positive semidefinite and we impose the following constraint:
\begin{equation*}
\langle U, X^T X\rangle \,=\,0
\end{equation*}
where $U\in\mathbb{S}^{n\times n}$. The objective is minimize the rank of $X^T X$ instead,
or equivalently to minimize $n-\langle I,U\rangle $.

\section{Constraint Qualification of the SDCMPCC Formulation   \label{sec:kkt}}

SDCMPCC problems are generally hard to solve and there have been discussions on potential methods to solve them \cite{wu2015properties,zhang2011convergence}, including relaxation and penalty methods. The original SDCMPCC formulation and all its variations fall into the genre of nonlinear semidefinite programming.
Most existing algorithms use the KKT conditions as criteria for checking local optimality,
and they terminate at KKT stationary points. The validity of KKT conditions at local optima can be guaranteed by constraint qualification.  However, as pointed out in \cite{dsye2010}, common constraint qualifications such as LICQ and Robinson CQ are violated for SDCMPCC.
The question arises as to whether any constraint qualification holds at the SDCMPCC formulation of a rank minimization problem. In this section we'll show that a constraint qualification called calmness holds at any
local optimum of the SDCMPCC formulation.
In this section, we assume only that
$\mathcal{C}$ is given
by the intersection of a closed convex cone and the set of solutions
to a collection of continuous inequalities.

\subsection{Introduction of Calmness}

Calmness was first defined by Clarke \cite{clarke1990optimization}.
If calmness holds then a first order KKT necessary condition holds at a local minimizer.
Thus, calmness plays the role of a constraint qualification,
although it involves both the objective function and the constraints.
It has been discussed in the context of conic optimization problems in
\cite{huang2006,jjYe15,zhai2014}, in addition to Ding et al.~\cite{dsye2010}.
Here, we give the definition from~\cite{clarke1990optimization},
adapted to our setting.
\begin{definition}
Let $K \subseteq \real^n$ be a convex cone.
Let $f:\real^q \rightarrow \real$, $h:\real^q\rightarrow\real^p$, $g:\real^q \rightarrow \real^m$,
and $G:\real^q\rightarrow\real^n$ be continuous functions.
A feasible point $\bar{x}$ to
the conic optimization problem
\begin{displaymath}
\begin{array}{lrcl}
\min_{x \in \real^q} & f(x) \\
\mbox{subject to} & g(x) & \leq & 0  \\
& h(x) & = & 0  \\
& G(x) & \in & K
\end{array}
\end{displaymath}
is {\em Clarke calm}
if there exist positive $\epsilon$ and $\mu$ such that
\begin{displaymath}
f(x) \, - \, f(\bar{x}) \, + \mu || (r,s,P) || \, \geq \, 0
\end{displaymath}
whenever $||(r,s,P)|| \leq \epsilon$, $||x-\bar{x}|| \leq \epsilon$,
and $x$ satisfies the following conditions:
\begin{displaymath}
h(x)+r=0, \, g(x)+s \leq 0, \, G(x) + P \in K.
\end{displaymath}
\end{definition}

The idea of calmness is that when there is a small perturbation in the constraints, the improvement in the objective value in a neighborhood of $\bar{x}$ must be bounded by some constant times the magnitude of perturbation.
\begin{theorem}   \label{thm:calmkkt}  \cite{dsye2010}
If calmness holds at a local minimizer $\bar{x}$ of (\ref{GSDCMPCC})
then the following first order necessary KKT conditions hold at~$\bar{x}$:
\begin{quote}
there exist multipliers $\lambda^h\in\mathbb{R}^p$, $\lambda^g\in\real^m$,
$\Omega^G\in\mathbb{S}^n_+$, $\Omega^H\in\mathbb{S}^n_+$,
and $\lambda^c\in\real$
such that the subdifferentials of the constraints and objective function
of (\ref{GSDCMPCC}) satisfy
\begin{displaymath}
\begin{array}{l}
0 \, \in \, \partial f(\bar{x}) \, + \, \partial \langle h,\lambda^h \rangle (\bar{x})
\, + \, \partial \langle g, \lambda^g \rangle (\bar{x})   \\ [5pt]  \qquad \qquad
\, + \, \partial \langle G, \Omega^G \rangle (\bar{x})
\, + \, \partial \langle H, \Omega^H \rangle (\bar{x})  %
\, + \, \lambda^c \partial c(\bar{x}) ,
\\ [5pt]
\lambda^g \geq 0, \, \langle g(\bar{x}), \lambda^g \rangle = 0,
\, \Omega^G \in \mathbb{S}^n_+, \, \Omega^H \in \mathbb{S}^n_+, \\ [5pt]
\langle \Omega^G, G(\bar{x}) \rangle = 0,
 \, \langle \Omega^H, H(\bar{x}) \rangle = 0.
\end{array}
\end{displaymath}
\end{quote}
\end{theorem}

In the framework of
general nonlinear programming, previous results \cite{lu2012relation} show that
the Mangasarian-Fromowitz constraint qualification (MFCQ) and
the constant-rank constraint qualification (CRCQ) imply local calmness.
When all the constraints are linear, CRCQ will hold.
However, in the case of SDCMPCC, calmness may not hold at locally optimal points.
Linear semidefinite programming programs are
a special case of SDCMPCC: take $H(x)$ identically equal to the zero matrix.
Even in this case, calmness may not hold.
For linear SDP, the optimality conditions in Theorem~\ref{thm:calmkkt}
correspond to primal and dual feasibility together with complementary slackness,
so for example any linear SDP which has a duality gap or where dual attainment fails will not satisfy calmness.
Consider the example below, where we show explicitly that calmness does not hold:
\begin{equation}\label{SDPTest}
\begin{array}{ll}
\displaystyle{
{\operatornamewithlimits{\mbox{minimize}}_{x_1,x_2}}
} & x_2\\ [0.15in]
\mbox{s.t}&G(x)\,=\,\left[\begin{array}{clcr}
x_2+1&0&0\\
0&x_1&x_2\\
0&x_2&0\end{array}\right]\,\succeq\,0
\end{array}
\end{equation}
It is trivial to see that any point $(x_1, 0)$ with $x_1 \geq 0$ is a global optimal point to the problem. However:
\begin{proposition}
Calmness does not hold at any point $(x_1, 0)$ with $x_1\geq 0$.
\end{proposition}
\begin{proof}
We will omit the case when $x_1>0$ and only show the proof for the case $x_1\,=\,0$. Take
\begin{equation*}
x_1\,=\,\,\delta\qquad\mbox{and} \qquad x_2\,=\, -\delta^2
\end{equation*}
As $\delta\rightarrow 0$, we can find a matrix:
\begin{equation*}
\begin{array}{ll}
M\,=\,\left[\begin{array}{clcr}
1 -\delta^2&0&0\\
0&\delta&-\delta^2\\
0&-\delta^2&\delta^3\end{array}\right]\,\succeq\,0
\end{array}
\end{equation*}
in the semidefinite cone and 
\begin{equation*}
||G(\delta,-\delta^2) - M|| = \delta^3.
\end{equation*}
However, the objective value at $(\delta,-\delta^2)$ is $-\delta^2$. Thus we have:
\begin{equation*}
\frac{f(x_1,x_2)\,-\,f(0,0)}{||G(\delta,-\delta^2) - M||}\,=\,
\frac{-\delta^2}{||G(\delta,-\delta^2) - M||}\,\leq\,\frac{-\delta^2}{\delta^3}\rightarrow -\infty
\end{equation*}
as $\delta\rightarrow 0$. It follows that calmness does not hold at the point $(0,0)$
since $\mu$ is unbounded.
\end{proof}

\subsection{Calmness of SDCMPCC Formulation}

In this part, we would like to show that in Problem (\ref{CompleRankPSD}), calmness holds for each pair $(X, U)$  with X feasible and U given by (\ref{AssignU}).
The variable $x$ in (\ref{GSDCMPCC}) is equal to the pair $(X,U)$ from~(\ref{CompleRankPSD}),
so $G(x)=X$ and $H(x)=U$.
We assume
\begin{displaymath}
\mathcal{C} \, = \, \mathbb{S}^n_+ \, \cap \{ X : g_i(X) \leq 0, i=1,\ldots,m_1\}
\end{displaymath}
for continuous functions $g_i(X)$;
these functions are incorporated into $g(x)$ in the formulation~(\ref{GSDCMPCC}).
Before presenting the propositions, we introduce the rank constrained problem
for any positive integer~$k$:
\begin{displaymath}
\begin{array}{llr}
\min_{X \in \mathbb{S}^n_+} & \psi(X)  \\
\mbox{subject to} & X \in \mathcal{C} & \qquad (RC(k))  \\
& \mbox{rank}(X) \, \leq \, k.
\end{array}
\end{displaymath}
\begin{proposition}   \label{prop:localopt}
Let $X$ be a local minimizer of~$(RC(k))$ for some choice of~$k$ and
let $U$ be given by (\ref{AssignU}).
Then $(X,U)$ is a local optimal solution in Problem (\ref{CompleRankPSD}).
Conversely, if
$(X,U)$ is a local optimal solution to (\ref{CompleRankPSD})
then $U$ is  given by (\ref{AssignU}) and
 $X$ is a local minimizer of~$(RC(k))$ with $k=\mbox{rank}(X)$.
\end{proposition}
\begin{proof}
The proposition follows from the fact that $\mbox{rank}(X')\geq \mbox{rank}(X)$ for all $X'$ close enough to $X$. 
\end{proof}
\begin{proposition}   \label{prop:calm}
For any local minimizer $X$ of~$(RC(k))$ for some choice of~$k$
with $U$ given by (\ref{AssignU}),
let $(\hat{X}, \hat{U})$ be a feasible point to the optimization problem below:
\begin{equation} \label{CompleRankPSDPerturbed}
\begin{array}{ll}
\displaystyle{
{\operatornamewithlimits{\mbox{minimize}}_{\hat{X}, \hat{U} \, \in \, \mathbb{S}^n}}
} &n\,-\,\langle I,\,\hat{U}\rangle \, + \, \psi(X) \\ [0.15in]
\mbox{subject to} & \hat{X}\,+\,p \in \tilde{\mathcal{C}}\\[0.15in]
& |\langle \hat{X},\,\hat{U}\rangle |\,\leq\,q\\[0.15in]
& \lambda_{min}(I\,-\,\hat{U})\,\geq\,-r\\[0.15in]
& \lambda_{min}(\hat{X})\,\geq\, -h_1\\[0.15in]
& \lambda_{min}(\hat{U})\,\geq\, -h_2
\end{array}
\end{equation}
where $p$, $q$, $r$, $h_1$ and $h_2$ are perturbations to the constraints and $\lambda_{min}(M)$
denotes the minimum eigenvalue of matrix~$M$.
Assume $X$ has at least one positive eigenvalue.
For $||(p,q,r,h_1,h_2)||$, $||X-\hat{X}||$, and $||U-\hat{U}||$
all sufficiently small, we have
\begin{equation}
\begin{array}{rcl}
\langle I,U\rangle -\langle I, \hat{U}\rangle &\geq & -\frac{2q}{\tilde{\lambda}}\,-\,
\left(n-\mbox{rank}(X)\right)\left(r+(1\,+\,r)\frac{2}{\tilde{\lambda}}\,h_1\right)  \\ [5pt]
&& \qquad \qquad -\, h_2 \left(\frac{4 }{\tilde{\lambda}}\,||X||_* - \mbox{rank}(X)\right)
\end{array}
\end{equation}
where $||X||_*$ is the nuclear norm of X and $\tilde{\lambda}$ is the smallest positive eigenvalue of X.
\end{proposition}
\begin{proof}
The general scheme is to determine a lower bound for  $ \langle I, U\rangle $ and an upper bound for $\langle I, \hat{U}\rangle $. A lower bound of $\langle I, U\rangle $ can be easily found by exploiting the complementarity constraints and its value is $n - \mbox{rank}(X)$. To find an upper bound of $\langle I, \hat{U}\rangle $, the approach we take is to fix $\hat{X}$ in Problem (\ref{CompleRankPSDPerturbed}) and estimate a lower bound for the objective value of the following problem:
\begin{equation} \label{XFixedPSDPerturbed}
\begin{array}{llr}
\displaystyle{
{\operatornamewithlimits{\mbox{minimize}}_{\tilde{U} \, \in \, \mathbb{S}^n}}
} &n\,-\,\langle I,\,\tilde{U}\rangle \\ [0.15in]
\mbox{subject to}%
&- \langle \hat{X},\,\tilde{U}\rangle \,\leq\,q, & \qquad y_1\\[0.15in]
& \langle \hat{X},\,\tilde{U}\rangle \,\leq\,q, &\qquad y_2\\[0.15in]
& I\,-\,\tilde{U}\,\succeq\,-r\, I, & \qquad \Omega_1\\[0.15in]
& \tilde{U}\,\succeq\, -h_2\, I, & \qquad \Omega_2
\end{array}
\end{equation}
where $y_1$, $y_2$, $\Omega_1$ and $\Omega_2$ are the Lagrangian multipliers for the corresponding constraints. It is obvious that $(\hat{X}, \hat{U})$  must be feasible to
Problem~(\ref{XFixedPSDPerturbed}).
We find an upper bound for $(I, \hat{U} )$ by finding
a feasible solution to the dual problem of Problem (\ref{XFixedPSDPerturbed}), which is:
\begin{equation} \label{XFixedPSDPerturbedDual}
\begin{array}{ll}
\displaystyle{
{\operatornamewithlimits{\mbox{maximize}}_{y_1,y_2\in \mathbb{R},\,\Omega_1, \Omega_2 \, \in \, \mathbb{S}^n}}
} &n\,+\,q\,y_1\,+\,q\,y_2\,-\,(1\,+\,r)\,trace(\Omega_1)\,-\,h_2\,trace(\Omega_2)\\ [0.15in]
\mbox{subject to}%
& -y_1\,\hat{X}\,+\,y_2\,\hat{X}\,-\,\Omega_1\,+\,\Omega_2\,=\,-I\\[0.15in]
&y_1,\,y_2\,\leq\,0\\[0.15in]
&\Omega_1,\,\Omega_2\,\succeq\,0
\end{array}
\end{equation}
We can find a lower bound on the dual objective value by looking at a tightened version, which is established by diagonalizing $\hat{X}$ by linear transformation and restricting the non-diagonal term of $\Omega_1$ and $\Omega_2$  to be 0. Let $\{f_i\}$, $\{g_i\}$ be the entries on the diagonal of $\Omega_1$ and $\Omega_2$ after the transformation respectively, and  $\{\hat{\lambda}_i\}$ be the  eigenvalues of $\hat{X}$. The tightened problem is:
\begin{equation} \label{TXFixedPSDPerturbedDual}
\begin{array}{ll}
\displaystyle{
{\operatornamewithlimits{\mbox{maximize}}_{y_1,y_2\in \mathbb{R},\,f,\,g \, \in \, \mathbb{R}^n}}
} &n\,+\,q\,y_1\,+\,q\,y_2\,-\,(1\,+\,r)\,\sum_i f_i \,-\,h_2\,\sum_i g_i\\ [0.15in]
\mbox{subject to}%
& -y_1\,\hat{\lambda}_i\,+\,y_2\,\hat{\lambda}_i\,-\,f_i\,+\,g_i\,=\,-1,\,\forall i\,=\,1\cdots n\\[0.15in]
&y_1,\,y_2\,\leq\,0\\[0.15in]
&f_i,\,g_i\,\geq\,0,\,\forall i\,=\,1\cdots n
\end{array}
\end{equation}
By proper assignment of the value of $y_1, y_2, f, g$, we can construct a feasible solution to the tightened problem and give a lower bound for the optimal objective of the dual problem.
Let $\{\lambda_i\}$ be the set of eigenvalues of~$X$,
with $\tilde{\lambda}$ the smallest positive eigenvalue, and set:
\begin{equation*}
y_1=0~ \mbox{and} ~y_2=-\frac{2}{\tilde{\lambda}}
\end{equation*}
For  f and g:
\begin{itemize}
\item if $\hat{\lambda}_i < \frac{\tilde{\lambda}}{2}$, take $f_i=1+y_2\hat{\lambda}_i$, $g_i=0$.
\item if $\hat{\lambda}_i \geq \frac{\tilde{\lambda}}{2}$, take $f_i=0$ and
$g_i=\frac{2}{\tilde{\lambda}}\hat{\lambda}_i-1$.
\end{itemize}
It is trivial to see that the above assignment will yield a feasible solution to
Problem~(\ref{TXFixedPSDPerturbedDual}) and hence a lower bound for the dual objective is:
\begin{equation}
n\,-\frac{2q}{\tilde{\lambda}}\,-\,\sum\limits_{\hat{\lambda}_i < \frac{\tilde{\lambda}}{2}}(1\,+\,r)(1+y_2\hat{\lambda}_i)
\,-\,h_2\,\sum\limits_{\hat{\lambda}_i \geq \frac{\tilde{\lambda}}{2}} (\frac{2}{\tilde{\lambda}}\hat{\lambda}_i-1)
\end{equation}
 By weak duality the primal objective value must be greater or equal to the dual objective value, thus:
\begin{equation*}
n - \langle I, \hat{U}\rangle \,\geq\,n\,-\frac{2q}{\tilde{\lambda}}\,-\,\sum\limits_{\hat{\lambda}_i < \frac{\tilde{\lambda}}{2}}(1\,+\,r)(1+y_2\hat{\lambda}_i)\,-\,h_2\,\sum\limits_{\hat{\lambda}_i \geq \frac{\tilde{\lambda}}{2}} (\frac{2}{\tilde{\lambda}}\hat{\lambda}_i-1).
\end{equation*}
Since we can write
\begin{equation*}
n - \langle I, U\rangle  \,=\,n\,-\,\sum\limits_{\lambda_i=0} 1,
\end{equation*}
it follows that for $||U-\hat{U}||$ sufficiently small we have
\begin{equation}\label{BoundOnU}
\begin{array}{ll}
(n-\langle I, \hat{U}\rangle )\,-\,(n-\langle I, U\rangle )  %
&\,\geq\,n\,-\frac{2q}{\tilde{\lambda}}\,-\,\sum\limits_{\hat{\lambda}_i < \frac{\tilde{\lambda}}{2}}(1\,+\,r)(1+y_2\hat{\lambda}_i)\,\\
&\qquad -\,h_2\,\sum\limits_{\hat{\lambda}_i \geq \frac{\tilde{\lambda}}{2}} (\frac{2}{\tilde{\lambda}}\hat{\lambda}_i-1)\,-\,(n\,-\,\sum\limits_{\lambda_i=0} 1)\\[0.25in]
&\,=\,-\frac{2q}{\tilde{\lambda}}\,-\,\sum\limits_{\hat{\lambda}_i < \frac{\tilde{\lambda}}{2}}(r+(1\,+\,r)y_2\hat{\lambda}_i)\\[0.15in]
&\qquad -\,h_2\,\sum\limits_{\hat{\lambda}_i \geq \frac{\tilde{\lambda}}{2}} (\frac{2}{\tilde{\lambda}}\hat{\lambda}_i-1).
\end{array}
\end{equation}
For $\hat{\lambda}_i < \frac{\tilde{\lambda}}{2}$, by the constraints $\hat{\lambda}_i\geq -h_1$ and setting $y_2 = -\frac{2}{\tilde{\lambda}}$, we have:
\begin{equation*}
r+(1\,+\,r)y_2\hat{\lambda}_i\,\leq\,r+(1\,+\,r)\frac{2 h_1}{\tilde{\lambda}}.
\end{equation*}
For $\hat{\lambda}_i \geq \frac{\tilde{\lambda}}{2}$, recall the definition for nuclear norm and we have:
\begin{equation*}
\sum\limits_{\hat{\lambda}_i \geq \frac{\tilde{\lambda}}{2}}\hat{\lambda}_i\,\leq\,2||X||_*
\end{equation*}
for $||X-\hat{X}||$ sufficiently small.
 Since there are exactly $n-\mbox{rank}(X)$ eigenvalues in $\hat{X}$ that converge to 0, we can simplify the above inequality(\ref{BoundOnU}) and have:
\begin{equation}
\begin{array}{ll}
(n-\langle I, \hat{U}\rangle )\,-\,(n-\langle I, U\rangle )  %
& \geq\,-\frac{2q}{\tilde{\lambda}}\,-\,(n-\mbox{rank}(X))(r+(1\,+\,r)\frac{2h_1}{\tilde{\lambda}})\\[0.15in]
&\qquad -\,h_2 \left(\frac{4 }{\tilde{\lambda}}\,||X||_* - \mbox{rank}(X)\right).
\end{array}
\end{equation}
Thus we can prove the inequality.
\end{proof}

There is one case that is not covered by Proposition~\ref{prop:calm},
namely that $X=0$. This is also calm, as we show in the next lemma.

\begin{lemma}  \label{lemma:calmX0}
Assume $X=0$ is feasible in (\ref{CompleRankPSD}), with $U$ given 
by~(\ref{AssignU}).
Let $(\hat{X}, \hat{U})$ be a feasible point to~(\ref{CompleRankPSDPerturbed}).
We have
\begin{displaymath}
(n-\langle I,\hat{U}\rangle ) \,  - \, (n-\langle I, U\rangle ) \, \geq \, -nr.
\end{displaymath}
\end{lemma}

\begin{proof}
Note that $\langle I,U\rangle = n$, since $X=0$ and $U$ satisfies~(\ref{AssignU}).
In addition, each eigenvalue of~$\hat{U}$ is no larger than~$1+r$,
so the result follows.
\end{proof}

\begin{proposition}\label{KCalmness}
Calmness holds at each $(X, U)$  provided
(i) $X$ is a local minimizer of~$(RC(k))$ for some choice of~$k$
and (ii) $U$ is given by (\ref{AssignU}).
\end{proposition}
\begin{proof}
This follows directly from Proposition~\ref{prop:calm}, Lemma~\ref{lemma:calmX0},
and the Lipschitz continuity assumption on~$\psi(X)$.
\end{proof}

It follows that any local minimizer of the SDCMPCC formulation of the rank
minimization problem is a KKT point. Note that 
no assumptions are necessary regarding~${\mathcal C}$.
\begin{theorem}  \label{thm:KKTsdcmpcc}
The KKT condition of Theorem~\ref{thm:calmkkt}
holds at each local optimum of Problem~(\ref{CompleRankPSD}).
\end{theorem}
\begin{proof}
This is a direct result from Theorem \ref{thm:calmkkt} and Propositions
\ref{prop:localopt} and~\ref{KCalmness}.
\end{proof}

The above results show that, similar to the exact complementarity formulation of $\ell_0$ minimization,
there are  many KKT stationary points in the exact SDCMPCC formulation of rank minimization,
so it is possible that an algorithm will terminate at some stationary point that might be far from
a global optimum.
As we have shown in the complementarity formulation for $\ell_0$ minimization
problem~\cite{FMPSWachter13},
a possible approach to overcome this difficulty is to relax the complementarity constraints.
In the following sections we investigate whether this approach works for the
SDCMPCC formulation.

\section{A Penalty Scheme for SDCMPCC Formulation   \label{sec:penalty}}

In this section and the following sections,
we present  a penalty scheme for the original SDCMPCC formulation.
The penalty formulation has the form:
\begin{equation} \label{PenaltyCompleRankPSD}
\begin{array}{ll}
\displaystyle{
{\operatornamewithlimits{\mbox{minimize}}_{X, U \, \in \, \mathbb{S}^n}}
} &n\,-\,\langle I,\,U\rangle  + \psi(X) + \rho \langle X,\,U\rangle \, =: \, f_{\rho}(X,U)  \\ [0.15in]
\mbox{subject to} & X \in \tilde{\mathcal{C}}\\[0.15in]
& 0\,\preceq\,I\,-\,U\\[0.15in]
& X\,\succeq\, 0\\[0.15in]
& U\,\succeq\,0
\end{array}
\end{equation}
We denote the problem as $SDCPNLP(\rho)$.
We discuss properties of the formulation in this section,
with an algorithm described in Section~\ref{sec:PALM}
and computational results given in Section~\ref{sec:comp}.
First, we note that it follows from standard results that a sequence of global minimizers
to (\ref{PenaltyCompleRankPSD}) converges to a global optimizer
of~(\ref{CompleRankPSD}); see Luenberger~\cite{LuenOld} for example.

From now on, we assume
\begin{displaymath}
\tilde{\mathcal{C}} \, = \, \{ X \in \mathbb{S}^n \, : \, \langle A_i, X \rangle = b_i, \, i=1,\ldots,m_2 \},
\end{displaymath}
where each $A_i \in \mathbb{S}^n$,
so $\tilde{\mathcal{C}}$ is an affine manifold.

\subsection{Constraint Qualification of Penalty Formulation}

The penalty formulation for the Rank Minimization problem is a nonlinear semidefinite program.
As with the exact SDCMPCC formulation, we would like to investigate whether algorithms for general 
nonlinear semidefinite programming problems can be applied to solve the penalty formulation.
As far as we know, most algorithms in nonlinear semidefinite programming use
first order KKT stationary conditions as the criteria for termination.
The KKT stationary condition at a local optimum of the penalty formulation is:
\begin{equation}\label{PenaltyKKTCondition}
\begin{array}{ll}
0\,\preceq\,U&\,\perp\,-I\,+\,\Psi \, + \, \rho X\,+\,Y\,\succeq 0 \\[0.15in]
0\,\preceq\,X&\,\perp\,\sum\lambda_i A_i\,+\,\rho\, U\,\succeq 0 \\[0.15in]
0\,\preceq\,Y&\,\perp\,I- U\,\succeq 0
\end{array}
\end{equation}
for $X\,\in\,\tilde{\mathcal{C}}$, where $\lambda \in \real^{m_2}$ are the dual multipliers corresponding
to the linear constraints  $Y \in \mathbb{S}^n_+$ are the dual multipliers
corresponding to the constraints $I-U \succeq 0$,
and $\Psi$ is a subgradient of $\psi(X)$.
Unfortunately, the counterexample below shows that the KKT conditions do not hold in
the penalty problem (\ref{PenaltyCompleRankPSD}) in general:
\begin{equation} \label{FailureofCalm2}
\begin{array}{ll}
\displaystyle{
{\operatornamewithlimits{\mbox{minimize}}_{X, U \, \in \, \mathbb{S}^3, x\,\in\,\mathbb{R}}}
} &3\,-\,\langle I,\,U\rangle \,+\,\rho\,\langle X, U\rangle \\ [0.15in]
\mbox{subject to} & X \,=\,\left[\begin{array}{clcr}
3+x&0&0\\
0&1-x&\frac{x}{2}\\
0&\frac{x}{2}&0\end{array}\right]\\[0.15in]
& 0\,\preceq\,I\,-\,U\\[0.15in]
& X\,\succeq\, 0\\[0.15in]
& U\,\succeq\,0
\end{array}
\end{equation}
Every feasible solution requires $x=0$.
It is obvious that if $\rho=0.5$ then the optimal solution to the above problem is:
\begin{equation*}
\bar{x}\,=\,0,\qquad \bar{X}\,=\,\left[\begin{array}{clcr}
3&0&0\\
0&1&0\\
0&0&0\end{array}\right]\qquad\mbox{and}\qquad \bar{U}\,=\,\left[\begin{array}{rrr}
0&0&0\\
0&1&0\\
0&0&1\end{array}\right]
\end{equation*}
and the optimal objective value is $1.5$.
However, there does not exist any KKT multiplier at this point.
We can show explicitly that calmness is violated at the current point. If we allow
$\lambda_{min}(X)\,\geq\,-h_1$ and $\lambda_{\max}(U) \leq 1+h_1$, then we can take
\begin{equation*}
\begin{array}{ll}
&x\,=\,\sqrt{h_1},\qquad X\,=\,\left[\begin{array}{clcr}
3+\sqrt{h_1}&0&0\\
0&1-\sqrt{h_1}&\sqrt{h_1}/2\\
0&\sqrt{h_1}/2&0\end{array}\right]  \\[0.25in]
&
\qquad \mbox{and}\qquad U\,=\,\left[\begin{array}{rrr}
0&0&0\\
0&1&-{h_1}\\
0&-{h_1}&1\end{array}\right]  .
\end{array}
\end{equation*}
It is obvious that $(x, X, U)\rightarrow(\bar{x},\bar{X},\bar{U})$ as $h_1\rightarrow 0$.
The resulting objective value at $(x, X, U)$ is
$1.5-0.5\sqrt{h_1}-0.5h_1^{1.5}$.
Thus, for small $h_1$, the difference in objective function value is $O(\sqrt{h_1})$,
while the perturbation in the constraints is only~$O(h_1)$,
so calmness does not hold.

Lack of constraint qualification indicates that algorithms such as the Augmented Lagrangian method
may not converge in general if applied to the penalty formulation.
However, if we enforce a Slater constraint qualification on the feasible region of~$X$,
we show below that calmness will hold in the penalty problem~(\ref{PenaltyCompleRankPSD})
at  local optimal points.
\begin{proposition}
Calmness holds at the local optima of the penalty formulation~(\ref{PenaltyCompleRankPSD}) if
$\cal{C}$ contains a positive definite matrix.
\end{proposition}
\begin{proof}
Since the Slater condition holds for  the feasible regions of both $X$ and $U$,
for each pair $(X_l, U_l)$ in the perturbed problem
we can find $(\tilde{X}_l, \tilde{U}_l)$ in $\tilde{C}\cap S^n_+$
with the distance between $(X_l, U_l)$ and $(\tilde{X}_l, \tilde{U}_l)$  bounded by
some constant times the magnitude of perturbation.

In particular, let $(\hat{X},\hat{U})$ be a strictly feasible solution to (\ref{PenaltyCompleRankPSD})
with minimum eigenvalue~$\delta>0$.
Let the minimum eigenvalue of $(X_l, U_l)$ be $-\epsilon<0$.
We construct
\begin{displaymath}
(\tilde{X}_l, \tilde{U}_l) \, = \, (X_l, U_l) \, + \,
\frac{\epsilon}{\delta+\epsilon} \left( (\hat{X},\hat{U}) \, - \, (X_l, U_l) \right)
\, \in \, %
S^n_+.
\end{displaymath}
Note that for $\epsilon << \delta$, we have
\begin{displaymath}
f_{\rho} (\tilde{X}_l, \tilde{U}_l) \, - \, f_{\rho}(X_l, U_l) \, = \, O(\epsilon),
\end{displaymath}
exploiting the Lipschitz continuity of~$\Psi(X)$.
As $(X_l, U_l)$ converges to $(\bar{X},\bar{U})$,
we also have $(\tilde{X}_l, \tilde{U}_l)\rightarrow (\bar{X},\bar{U})$,
so by the local optimality of $(\bar{X},\bar{U})$
we can give a bound on the optimal value of the perturbed problem
and the statement holds.
\end{proof}
It follows that the KKT conditions will hold at local minimizers for
the penalty formulation.
\begin{proposition}
The first order KKT condition holds at local optimal solutions for the penalty formulation
(\ref{PenaltyCompleRankPSD}) if the Slater condition holds for the feasible region
$\tilde{C}\cap S^n_+$ of~$X$.
\end{proposition}

\subsection{Local Optimality Condition of Penalty Formulation}

\subsubsection{Property of KKT Stationary Points of Penalty Formulation}

The KKT condition in  the penalty formulation works in a similar way with some thresholding
methods~\cite{candes1,cambierabsil2016,vandereycken2013,weiCCL2016b}.
The objective function not only counts the number of zero eigenvalues,
but also the number of eigenvalues below a certain threshold,
as illustrated in the following proposition.
\begin{proposition}
Let $(\bar{X},\bar{U})$ be a local optimal solution to the penalty formulation. Let $\sigma_i$ be an eigenvalue of $\bar{U}$ and $v_i$ be a corresponding eigenvector of norm one.
It follows that:
\begin{itemize}
\item If $\sigma_i \,=\,1$, then $v_i^T (\Psi + \rho \bar{X}) v_i\,\leq\,1$.
\item If $\sigma_i \,=\,0$, then $v_i^T (\Psi + \rho \bar{X}) v_i\,\geq\,1$.
\item if $0\,<\,\sigma_i\,<\,1$, then $v_i^T (\Psi + \rho \bar{X}) v_i\,=\,1$ 
\end{itemize}
\end{proposition}
\begin{proof}
If $\sigma_i \,=\,1$, since $v_i$ is an eigenvector of U, by the complementarity in the KKT condition it follows that:
\begin{displaymath}
v_i^T(-I\,+\,\Psi \, + \, \rho X\,+\,Y)v_i\,=\,0.
\end{displaymath}
As $v_i^T Y v_i \,\geq\,0$, we have $v_i^T (\Psi + \rho \bar{X}) v_i\,\leq\,1$.

If $\sigma_i\,=\,0$, then $v_i$ is an eigenvector of $I-U$ with eigenvalue~1.
By the complementarity of $I -U$ and $Y$ we have $v_i^T Y v_i = 0$ and
\begin{displaymath}
v_i^T(-I\,+\,\Psi \, + \, \rho \bar{X}\,+\,Y)v_i\,=\,-1\,+\, v_i^T (\Psi + \rho \bar{X}) v_i\,\geq\,0.
\end{displaymath}
The above inequality is satisfied if and only if $v_i^T (\Psi + \rho \bar{X}) v_i \geq 1$.

If $0\,<\,\sigma_i\,<\,1$,then $v_i$ is an eigenvector of $I-U$ corresponding to an eigenvalue in $(0,1)$. The complementarity in KKT condition implies   that $v_i^T Y v_i \,=\, 0\,=\,
v_i^T(-I\,+\,\Psi \, + \, \rho \bar{X}\,+\,Y)v_i$. It follows that $v_i^T (\Psi + \rho \bar{X}) v_i\,=\,1$. 
\end{proof}

Using the proposition above, we can show the equivalence between the stationary points of the SDCMPCC formulation and the penalty formulation.
\begin{proposition}   \label{prop:exactpenalty}
If $(\bar{X},\bar{U})$ is a stationary point  of the SDCMPCC formulation
with corresponding subgradient~$\Psi$
then it is a stationary point of the penalty formulation
if the penalty parameter $\rho$ is sufficiently large.
\end{proposition}
\begin{proof}
Choose $\rho$ so that the 
minimum positive eigenvalue of $\Psi + \rho \bar{X}$
is strictly greater than~$1$.
By setting $\lambda = 0$ and with a proper assignment of $Y$ we can see that first order optimality
condition (\ref{PenaltyKKTCondition}) holds at $(\bar{X},\bar{U})$ for such a choice of~$\rho$,
thus $(\bar{X},\bar{U})$ is a KKT stationary point for the penalty formulation.
\end{proof}

\subsubsection{Local Convergence of KKT Stationary Points}

We would like to investigate whether local convergence results can be established for
the penalty formulation, that is, whether the limit points of KKT stationary points of the
penalty scheme are KKT stationary points of the SDCMPCC formulation.
Unfortunately, local convergence does not hold for the penalty formulation,
although the limit points are feasible in the original SDCMPCC formulation.
\begin{proposition}
Let $(X_k, U_k)$ be a KKT stationary point of the penalty scheme with 
subgradient $\Psi_k$ and
penalty parameter $\{\rho_k\}$. As $\rho_k \rightarrow \infty$, any limit point $(\bar{X},\bar{U})$ of the sequence $\{(X_k, U_k)\}$ is a feasible solution to the original problem.
\end{proposition}
\begin{proof}
The proposition can be verified by contradiction.
Note that the norm of $\Psi_k$ is bounded since $\psi(X)$ is Lipschitz continuous.
If the Frobenius inner product of
$\bar{X}$ and $\bar{U}$ is greater than 0,
then when k is large enough we have:
\begin{displaymath}
\begin{array}{ll}
\langle U_k, -I_k\,+\,\Psi_k \, + \, \rho_k X_k\,+\,Y_k\rangle &\,\geq\,-\langle U_k, I\rangle \, +
\, \langle U_k, \Psi_k \rangle \, +\,\rho_k \langle X_k, U_k\rangle \, > \, 0
\end{array}
\end{displaymath}
which violates the complementarity in the KKT conditions of the penalty formulation.
\end{proof}

We can show that a limit point may not be a KKT stationary point. Consider the following problem:
\begin{equation}
\begin{array}{ll}
\displaystyle{
{\operatornamewithlimits{\mbox{minimize}}_{X,U \, \in \, \mathbb{S}^3_+}}
} &n\,-\,\langle I,U\rangle \\[0.15in]
\mbox{subject to} ~&X_{11}\,=\,1,~0\,\preceq\,U\,\preceq\,I~\mbox{and}~\langle U,X\rangle \,=\,0
\end{array}
\end{equation}
The penalty formulation takes the form:
\begin{equation}
\begin{array}{ll}
\displaystyle{
{\operatornamewithlimits{\mbox{minimize}}_{X,U \, \in \, \mathbb{S}^2_+}}
} &n\,-\,\langle I,U\rangle \,+\,\rho_k\langle X,U\rangle \\[0.15in]
\mbox{subject to}  ~&X_{11}\,=\,1,~0\,\preceq\,U\,\preceq\,I
\end{array}
\end{equation}
Let $X_k$ and $U_k$ take the value:
\begin{equation*}
X_k\,=\,\left[\begin{array}{clcr}
1&0\\
0&\frac{2}{\rho_k}\end{array}\right],\,
U_k\,=\,\left[\begin{array}{clcr}
0&0\\
0&0\end{array}\right],\,
\end{equation*}
so $(X_k, U_k)$ is a KKT stationary point for the penalty formulation. However, the limit point:
\begin{equation*}
\bar{X}\,=\,\left[\begin{array}{clcr}
1&0\\
0&0\end{array}\right],\,
U_k\,=\,\left[\begin{array}{clcr}
0&0\\
0&0\end{array}\right],\,
\end{equation*}
is not a KKT stationary point for the original SDCMPCC formulation.

\section{Proximal Alternating Linearized Minization  \label{sec:PALM}}

The proximal alternating linearized minimization (PALM) algorithm of
Bolte et al.~\cite{bolte2014proximal} is used to solve a wide class
of nonconvex and nonsmooth problems of the form
\begin{equation}    \label{eq:bolte_prob}
\mbox{minimize}_{x \in \real^n, y \in \real^m}~ \Phi(x,y)\,:=\,f(x)\,+\,g(y)\,+\,H(x,y)
\end{equation}
where $f(x)$, $g(y)$ and $H(x,y)$ satisfy smoothness and continuity assumptions:
\begin{itemize}
\item $f:\mathbb{R}^n \rightarrow (-\infty, +\infty]$ and $g:\mathbb{R}^m \rightarrow (-\infty, +\infty]$ are proper and lower semicontinuous functions.
\item H: $\mathbb{R}^n \times \mathbb{R}^m\rightarrow \mathbb{R}$ is a $C^1$ function.
\end{itemize}
No convexity assumptions are made.
Iterates are updated using a proximal map
with respect to a function $\sigma$ and weight parameter $t$:
\begin{equation}
\mbox{prox}_t^{\sigma} (\hat{x}) \,
= \,\mbox{argmin}_{x \in\mathbb{R}^d}\left\{\sigma(x) \, + \, \frac{t}{2}||x\,-\,\hat{x}||^2 \right\} .
\end{equation}
When $\sigma$ is a convex function, the objective is strongly convex and the map
returns a unique solution.
The PALM algorithm is given in Procedure~\ref{proc:palm}.
\begin{algorithm}
\SetAlgorithmName{Procedure}{}

\SetKwInOut{Input}{input}\SetKwInOut{Output}{output}

\Input{Initialize with any $(x^0,y^0)\in \mathbb{R}^n \times\mathbb{R}^m$.
Given Lipschitz constants $L_1(y)$ and $L_2(x)$ for the partial gradients of $H(x,y)$
with respect to $x$ and~$y$.}
\Output{Solution to (\ref{eq:bolte_prob}).}
\BlankLine
For  $k = 1,2,\cdots$ generate a sequence $\{x^k ,y^k\}$ as follows:
\begin{itemize}
\item Take $\gamma_1 > 1$, set $c_k = \gamma_1 L_1(y^k)$ and compute:
$x^{k+1} \in \mbox{prox}_{c_k}^f(x^k\,-\,\frac{1}{c_k}\nabla_x H(x^k,y^k))$.
\item Take $\gamma_2 > 1$, set $d_k = \gamma_2 L_2(x^k)$ and compute:
$y^{k+1} \in \mbox{prox}_{d_k}^g(y^k\,-\,\frac{1}{d_k}\nabla_y H(x^k,y^k))$.
\end{itemize}
\caption{PALM algorithm}\label{proc:palm}
\end{algorithm} 
It was shown in \cite{bolte2014proximal} that it converges
to a stationary point of $\Phi(x,y)$ under the following assumptions on the functions:
\begin{itemize}
\item $\inf_{\mathbb{R}^m\times \mathbb{R}^n} \Phi>-\infty$, $\inf_{\mathbb{R}^n} f> -\infty$ and $\inf_{\mathbb{R}^m}g > -\infty$.

\item The partial gradient $\nabla_x H(x, y)$ is globally Lipschitz with moduli~$L_1(y)$, so:
\begin{equation*}
||\nabla_x H(x_1, y) - \nabla_x H(x_2, y)||\,\leq\,L_1(y) ||x_1 - x_2 ||, \,
\forall x_1, x_2 \in \mathbb{R}^n.
\end{equation*}
Similarly, the partial gradient $\nabla_y H(x, y)$ is globally Lipschitz with moduli~$L_2(x)$.

\item For $i = 1, 2$, there exists $\lambda_i^-,\lambda_i^+$ such that:
\begin{itemize}
\item $\inf\{L_1(y^k): k\in\mathbb{N} \}\geq \lambda_1^-$ and $\inf\{L_2(x^k): k\in\mathbb{N}\} \geq \lambda_2^-$
\item $\sup\{L_1(y^k): k\in\mathbb{N} \}\leq \lambda_1^+$ and $\sup\{L_2(x^k): k\in\mathbb{N}\} \leq \lambda_2^+$.
\end{itemize}

\item $\nabla H$ is continuous on bounded subsets of $\mathbb{R}^n \times \mathbb{R}^m$,
i.e, for each bounded subset $B_1 \times B_2$ of $\mathbb{R}^n \times \mathbb{R}^m$
there exists $M>0$ such that for all $(x_i, y_i) \in B_1 \times B_2$, $i = 1,2$:
\begin{equation}
\begin{array}{ll}
&||(\nabla_x H(x_1,y_1) - \nabla_x H(x_2,y_2),\nabla_y H(x_1,y_1) - \nabla_y H(x_2,y_2), )||\\ [5pt]
& \qquad \qquad \leq \, M||(x_1-x_2, y_1 - y_2)||.
\end{array}
\end{equation}
\end{itemize}

The PALM method can be applied to the penalty formulation of SDCMPCC formulation of
rank minimization (\ref{PenaltyCompleRankPSD})
with the following assignment of the functions:
\begin{equation*}
\begin{array}{ll}
&f(X)\,=\, \frac{1}{2} \rho_X ||X||_F^2\,+\,\psi(X) \, + \, \mathcal{I}(X\in\mathcal{C})\\
&g(U)\,=\,n\,-\,\langle I,\,U\rangle \,+\,\mathcal{I}(U\in\{0 \preceq U\preceq I\})\\
&H(X,\,U)\,=\, \rho \langle X,\,U\rangle ,
\end{array}
\end{equation*}
where
$\mathcal{I}(.)$ is an indicator function taking the value 0 or~$+\infty$, as appropriate.
Note that we have added a regularization term $||X||_F^2$ to the objective function.
When the feasible region for $X$ is bounded, the assumptions required for the convergence of the PALM
procedure hold for the penalty formulation of SDCMPCC. 
\begin{proposition}
The function $\Phi(X,U) = f(X)\,+\,g(U)\,+\,H(X,U)$ is bounded below for $X \in \mathbb{S}^n$ and $U \in \mathbb{S}^n$
if $\psi(X)$ is bounded below.
\end{proposition}
\begin{proof}
Since the eigenvalues of $U$ are bounded by 1, and the Frobenius norm of $X$ and $U$
must be nonnegative, the statement is obvious.
\end{proof}

\begin{proposition}
If the feasible region of $X$ is bounded then $H(X,U)$
is globally Lipschitiz. 
\end{proposition}
\begin{proof}
The gradient of $H(X,U)$ is:
\begin{equation*}
(\nabla_X H(X,U), \nabla_U H(X,U))\,:=\,\rho(U, X)
\end{equation*}
The statement results directly from the boundedness of the feasible region of $X$ and~$U$.
\end{proof}

\begin{proposition}
$\nabla H(X,U) = \rho(U,\,X)$ is continuous on bounded subsets of
$\mathbb{S}^n \times \mathbb{S}^n$.
\end{proposition}

The proximal subproblems in Procedure~\ref{proc:palm} are both convex
quadratic semidefinite programs. The update to $U^{k+1}$ has a closed
form expression based on the eigenvalues and eigenvectors
of $U^k - (\rho/d_k)X^k$.
Rather than solving the update problem for $X^{k+1}$ directly using, for
example, CVX~\cite{dcp06}, we found it more effective to solve the dual
problem using Newton's method, with a conjugate gradient approach to
approximately solve the direction-finding subproblem.
This approach was motivated by work of Qi and Sun~\cite{qi2006quadratically}
on finding nearest correlation matrices.
The structure of the Hessian for our problem is such that the
conjugate gradient approach is superior to a direct factorization approach,
with matrix-vector products relatively easy to calculate compared to
formulating and factorizing the Hessian.
The updates to $X$ and $U$ are discussed in Sections \ref{sec:updateX}
and~\ref{sec:updateU}, respectively.
First, we discuss accelerating the PALM method.

\subsection{Adding Momentum Terms to the PALM Method\label{sec:fastPALM}}

One downside for proximal gradient type methods are their slow convergence rates.
Nesterov~\cite{nesterov2013introductory,nesterov22} proposed accelerating gradient methods for convex
programming by adding a momentum term.
The accelerated algorithm has a quadratic
convergence rate, compared with sublinear convergence rate of the normal gradient method.
Recent
accelerated proximal gradient methods include~\cite{schmidt2011convergence,toh2010accelerated}.

Bolte et al.~\cite{bolte2014proximal} showed that the PALM proximal gradient method
can be applied to nonconvex programs under certain assumptions and the method will converge
to a local optimum. The question arises as to whether there exists an accelerated version in
nonconvex programming. Ghadimi et al.~\cite{ghadimi2016accelerated} presented an
accelerated gradient method for nonconvex nonlinear and stochastic programming,
with quadratic convergence to a limit
point satisfying the first order KKT condition.

There are various ways to set up the momentum term \cite{li2015accelerated,lin2015accelerated}.
Here we adopted the following strategy while updating $x^k$ and~$y^k$:
\begin{equation}
x^{k+1} \in \mbox{prox}_{c_k}^f(x^k\,-\,\frac{1}{c_k}\nabla_x H(x^k,y^k)\,+\,\beta^k(x^k - x^{k-1})).
\end{equation}
\begin{equation}
y^{k+1} \in \mbox{prox}_{d_k}^g(y^k\,-\,\frac{1}{d_k}\nabla_y H(x^k,y^k)\,+\,\beta^k(y^k - y^{k-1})),
\end{equation}
where $\beta^k \,=\,\frac{k-2}{k+1}$ (borrowing from~\cite{nesterov22}).
We refer to this as the Fast PALM method.

\subsection{Updating $X$\label{sec:updateX}}

Assume $\mathcal{C}\,=\,\{X \in \mathbb{S}^n \, : \, \mathcal{A}(X)\,=\,b, \,X\succeq 0\}$ and
the Slater condition holds for~$\mathcal{C}$.
The proximal point $X$ of $\tilde{X}$, or $\mbox{prox}_{c_k}^{f}(\tilde{X})$
can be calculated as the optimal solution to the problem:
\begin{equation} \label{proximalX}
\begin{array}{ll}
\displaystyle{
{\operatornamewithlimits{\mbox{minimize}}_{X \, \in \, \mathbb{S}^{n}}}
} & \rho_X ||X||_F^2\,+\, \psi(X) \, + \, c_k||X-\tilde{X}||_F^2\\ [0.15in]
\mbox{subject to} &\mathcal{A}(X)\,=\,b\\
&X\,\succeq\,0
\end{array}
\end{equation}
The objective can be replaced by:
\begin{equation*}
||X\,-\,\frac{c_k}{c_k+\rho_X}\tilde{X}||_F^2 \, + \, \frac{1}{\rho_X} \psi(X).
\end{equation*}
With the Fast PALM method, we use
\begin{displaymath}
\tilde{X} \, = \, X^k \, - \, \frac{\rho}{c_k} U^k \, + \, \beta^k ( X^k - X^{k-1}).
\end{displaymath}
We observed that the structure of the subproblem to get the proximal point is very similar to the
nearest correlation matrix problem when $\psi(X) \equiv 0$.
Qi and Sun~\cite{qi2006quadratically} showed that for the nearest correlation matrix problem,
a semismooth Newton's method is numerically very efficient compared to other existing methods,
and it achieves a quadratic convergence rate if the Hessian at the optimal solution
is positive definite.
Provided Slater's condition holds for the feasible region of the subproblem and its dual program,
strong duality holds and instead of solving the primal program,
we can solve the dual program which has the following form:
\begin{equation*}
\mbox{min}_{y} \theta(y) :=\,\frac{1}{2}||(G\,+\mathcal{A}^*y)_+||_F^2\,-\,b^T y
\end{equation*}
where $(M)_+$ denotes the projection of the symmetric matrix $M \in \mathbb{S}^n$
onto the cone of positive semidefinite matrices~$\mathbb{S}^n_+$,
and
\begin{displaymath}
G \, = \,  \frac{c_k}{c_k+\rho_X}\tilde{X}  .
\end{displaymath}
One advantage of the dual program over the primal program is that the dual program is unconstrained.
Newton's method can be applied to get the solution $y^*$
which satisfies the first order optimality condition:
\begin{equation}   \label{eq:focXupdate}
F(y) \, := \, \mathcal{A}\,(G\,+\,\mathcal{A}^*y)_+\,=\,b.
\end{equation}
Note that $\nabla \theta(y)=F(y)-b$.
The algorithm is given in Procedure~\ref{proc:focXupdate}.
\begin{algorithm}
\SetAlgorithmName{Procedure}{}

\SetKwInOut{Input}{input}\SetKwInOut{Output}{output}

\Input{Given $y^0, \eta \in (0,1), \rho, \sigma$. Initialize $k=0$.}
\Output{Solution to (\ref{eq:focXupdate}).}
\BlankLine
For each $k = 0, 1,2,\cdots$ generate a sequence $y^{k+1}$ as follows:
\begin{itemize}
\item Select $V_k \in \partial F(y^k)$ and apply the
conjugate gradient method \cite{hestenes1952methods} to find an approximate solution $d^k$ to:
\begin{equation*}
\nabla \theta(y^k)\,+\,V_k d\,=\,0
\end{equation*}
such that:
\begin{equation*}
||\nabla\theta(y^k)\,+\,V_kd^k||\,\leq\,\eta_k ||\nabla\theta(y^k)|| .
\end{equation*}
\item Line Search. Choose the smallest nonnegative integer $m_k$ such that:
\begin{equation*}
\theta(y^k\,+\,\rho^m d^k)\,-\,\theta(y^k)\,\leq\,\sigma \rho^m \nabla\theta(y^k)^T d^k .
\end{equation*}
\item Set $t_k:=\rho^{m_k}$ and $y^{k+1}:= y^k\,+\,t_k d^k$.
\end{itemize}
\caption{Newton's method to solve (\ref{eq:focXupdate})\label{proc:focXupdate}}
\end{algorithm}

In each iteration, one key step is to construct the Hessian matrix $V_k$. Given the eigenvalue decomposition
\begin{equation*}
G\,+\,\mathcal{A}^*y\,=\,P \Lambda P^T ,
\end{equation*}
let $\alpha, \beta, \gamma$ be the sets of indices corresponding to positive, zero
and negative eigenvalues $\lambda_i$ respectively. Set:
\begin{equation*}
M_y\,=\,\left[\begin{array}{clcr}
E_{\alpha \alpha}&E_{\alpha \beta}&(\tau_{ij})_{i\in\alpha,j\in\beta}\\
E_{\beta \alpha}&0&0\\
(\tau_{ji})_{i\in\alpha,j\in\beta}&0&0\end{array}\right]
\end{equation*}
where all the entries in $E_{\alpha\alpha}$, $E_{\alpha\beta}$ and $E_{\beta\alpha}$ take value 1
and
\begin{equation*}
\tau_{ij}\,:=\,\frac{\lambda_i}{\lambda_i\,-\,\lambda_j}, i\,\in\,\alpha,~j\,\in\,\gamma
\end{equation*}
The operator $V_y:\mathbb{R}^m\rightarrow \mathbb{R}^m$ is defined as:
\begin{equation*}
V_y h\,=\,\mathcal{A} \left( P(M_y\,\circ \,(P^T (\mathcal{A}^*h) P))\,P^T \right) ,
\end{equation*}
where $\circ$ denotes the Hadamard product.
Qi and Sun~\cite{qi2006quadratically} showed
\begin{proposition}
The operator $V_y$ is positive semidefinite, with
$\langle h, V_yh\rangle \,\geq\,0$ for any $h\in\mathbb{R}^m$.
\end{proposition}

Since positive definiteness of $V_y$ is required for the conjugate gradient method,
a perturbation term is added in the linear system:
\begin{equation*}
(V_y\,+\,\epsilon I) d^k\,+\,\nabla \theta(y^k)\,=\,0 .
\end{equation*}
After getting the dual optimal solution $y^*$, the primal optimal solution $X^*$ can be recovered by:
\begin{equation*}
X^*\,=\,(G\,+\,\mathcal{A}^*y^*)_+ .
\end{equation*}

\subsection{Updating $U$\label{sec:updateU}}

The subproblem to update $U$ is:
\begin{equation*}
\begin{array}{ll}
\displaystyle{
{\operatornamewithlimits{\mbox{minimize}}_{U \, \in \, \mathbb{S}^{n}}}
} & n\,-\,\langle I,\,U\rangle \,+\,\frac{d_k}{2} ||U-\tilde{U}||_F^2\\ [0.15in]
\mbox{subject to} & 0\,\preceq\,U\,\preceq\,I,
\end{array}
\end{equation*}
with
\begin{displaymath}
\tilde{U} \, = \,  U^k \, - \, \frac{\rho}{d_k} X^k \, + \, \beta^k ( U^k - U^{k-1})
\end{displaymath}
with the Fast PALM method.
The objective is equivalent to minimizing
$||U\,-\,(\tilde{U} + \frac{1}{d_k}I)||_F^2$.
An analytical solution can be found for this problem.
Given the eigenvalue decomposition of
$\tilde{U} + \frac{1}{d_k}I$:
\begin{equation*}
\tilde{U} + \frac{1}{d_k}I
\,=\,\sum_{i=1}^n \sigma^{\tilde{U}}_i v_i v_i^T,
\end{equation*}
the optimal solution $U^*$ is:
\begin{equation*}
U^*\,=\, \sum_{i=1}^n\sigma^{U^*}_i v_i v_i^T
\end{equation*}
where the eigenvalue $\sigma^{U^*}_i$ takes the value:
\begin{equation*}
\sigma^{U^*}_i\,=\,\left\{
\begin{array}{ll}
0 & \mbox{if}~~\sigma^{\tilde{U}}_i \,<\,0\\
\sigma^{\tilde{U}}_i & \mbox{if}~~0\,\leq\,\sigma^{\tilde{U}}_i \,\leq\,1\\
1 & \mbox{if}~~\sigma^{\tilde{U}}_i \,> \,1
\end{array}
\right.
\end{equation*} 
Note also that
\begin{displaymath}
U^* \, = \, I \, - \, \sum_{i=1}^n (1-\sigma^{{U}^*}_i) \, v_iv_i^T.
\end{displaymath}
It may be more efficient to work with this representation if there are many more
eigenvalues at least equal to one as opposed to less than zero.

\section{Test Problems  \label{sec:comp}}

Our experiments included tests on coordinate recovery
problems~\cite{aapwolkowicz2011,ding1,laurent8,pongTsengMP,so1}.
In these problems, the distances between items in~$\real^3$ are given and it is necessary to recover
the positions of the items.
Given an incomplete Euclidean distance matrix $D=(d_{ij}^2)$, where:
\begin{equation*}
d_{ij}^2\,=\,\mbox{dist}(x_i,x_j)^2\,=\,{(x_i-x_j)^T(x_i-x_j)},
\end{equation*}
we want to recover the coordinate $x_i, i\,=\,1,\cdots,n$. Since the coordinate is in 3-dimensional space,
$x_i, i =1,\cdots,n $ is a $1\times 3$ vector.
Let $X\,=\,(x_1,x_2,\cdots,x_n)^T\in \mathbb{R}^{n\times 3}$.
The problem turns into recovering the matrix $X$. One way to solve the problem is to lift $X$ by
introducing $B = X X^T$ and we would like to find B that satisfies
\begin{equation*}
B_{ii}\,+\,B_{jj}\,-\,2\,B_{ij}\,=\,d_{ij}^2, ~~\forall (i,j)\,\in\,\Omega
\end{equation*}  
where $\Omega$ is set of pairs of points whose distance has been observed.
Since $X$ is of rank 3, the rank of the symmetric psd matrix $B$ is 3,
and so we seek to minimize the rank of $B$ in the objective.

\begin{figure}
  \centering
    \includegraphics[width=1\linewidth]{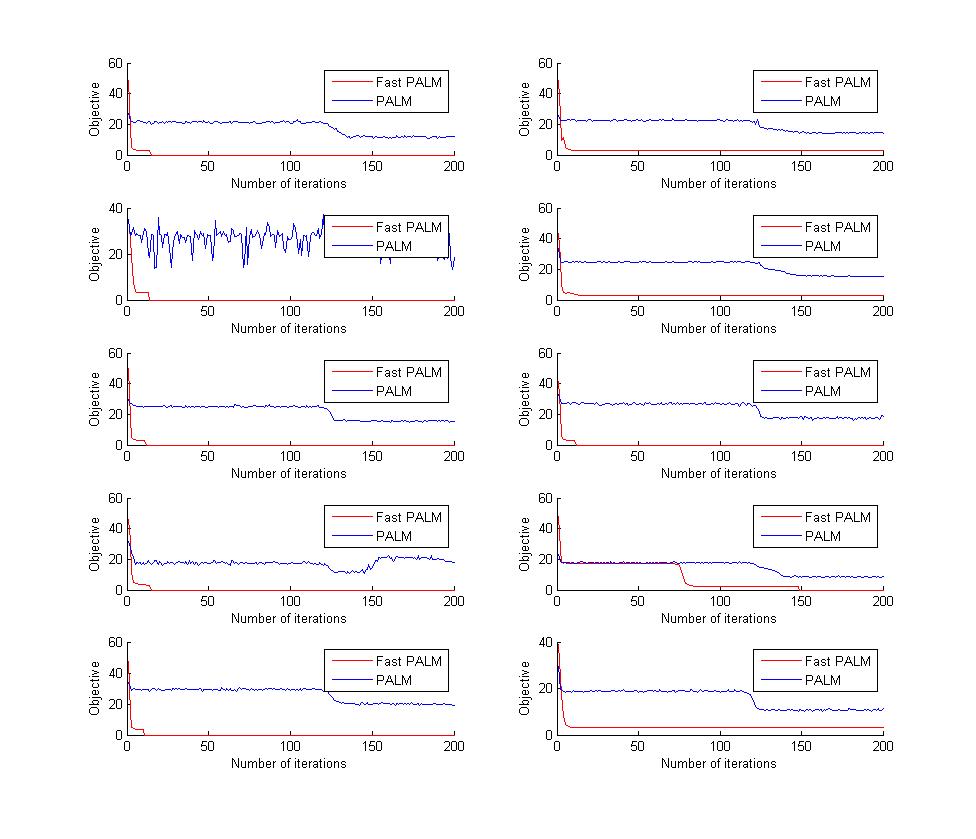}
    \caption[Computational Results of Case 1-10]{\small 
    Computational Results for 10 distance
    recovery instances. The objective value for the two methods at each iteration is plotted.
    If the algorithm  terminates before the iteration limit, the objective value after termination is 0
    in the figure.}
    \label{FastPALM1}
\end{figure}
\begin{figure}
  \centering
    \includegraphics[width=1\linewidth]{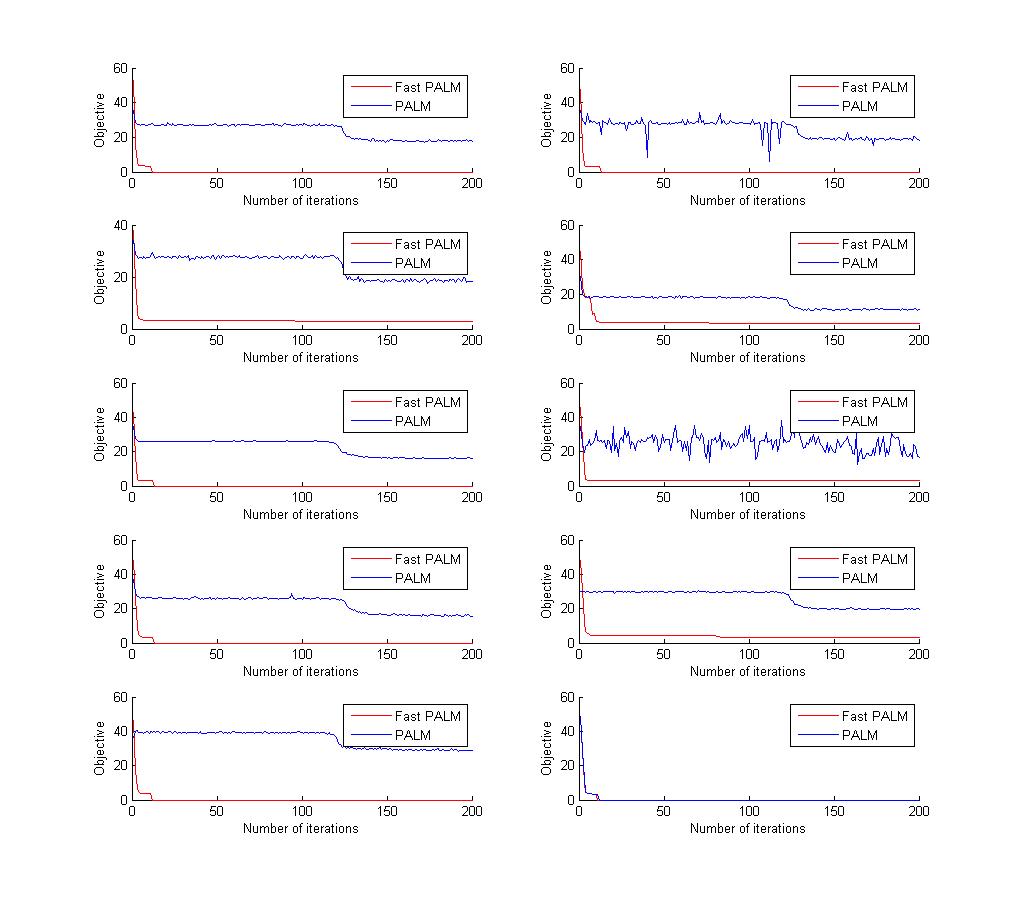}
    \caption[Computational Results of Case 11-20]{\small  Computational Results of Case 11-20}
    \label{FastPALM2}
\end{figure}
We generated 20 instances and in each instance we randomly sampled
150 entries from a $50 \times 50$
distance matrix. We applied the PALM method and the Fast PALM method to solve the problem.
For each case, we limit the maximum number of iterations to 200.
Figures \ref{FastPALM1} and~\ref{FastPALM2} each show the results on 10 of the instances.
As can be seen,
the Fast PALM approach dramatically speeds up the algorithm.
The computational tests were conducted using Matlab R2013b running in Windows 10
on an Intel Core i7-4719HQ CPU @2.50GHz with 16GB of RAM.

We  compared the rank of the solutions returned by the fast PALM method for
(\ref{PenaltyCompleRankPSD})
with the solution returned by the convex nuclear norm approximation to~(\ref{RankPSD}).
Note that when we calculate the rank of the resulting $X$ in both the convex approximation and
the penalty formulation, we count the number of eigenvalues above the threshold 0.01. 
\begin{figure}
   \begin{center}
    \includegraphics[scale=0.5]{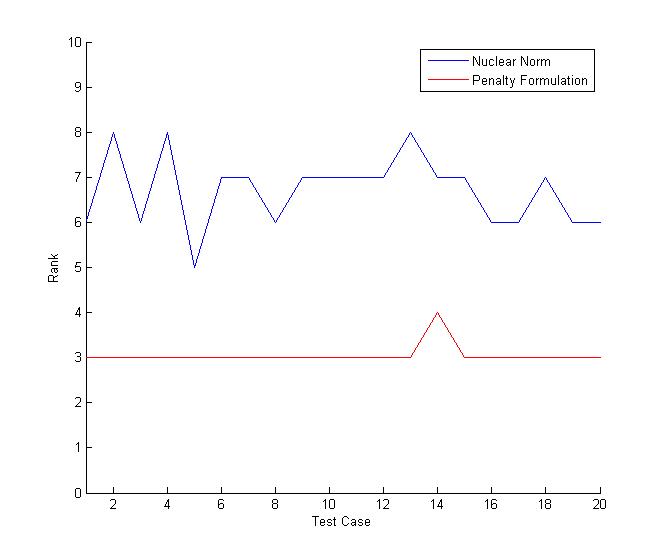}
    \end{center}
        \caption[Rank of Solutions When 150 Distances are Sampled]{\small  Rank of solutions
        when 150 distances are sampled for 20 instances with $150 \times 150$ distance matrix.}
    \label{Rank150}
\end{figure}
\begin{figure}
\begin{center}
     \includegraphics[scale=0.5]{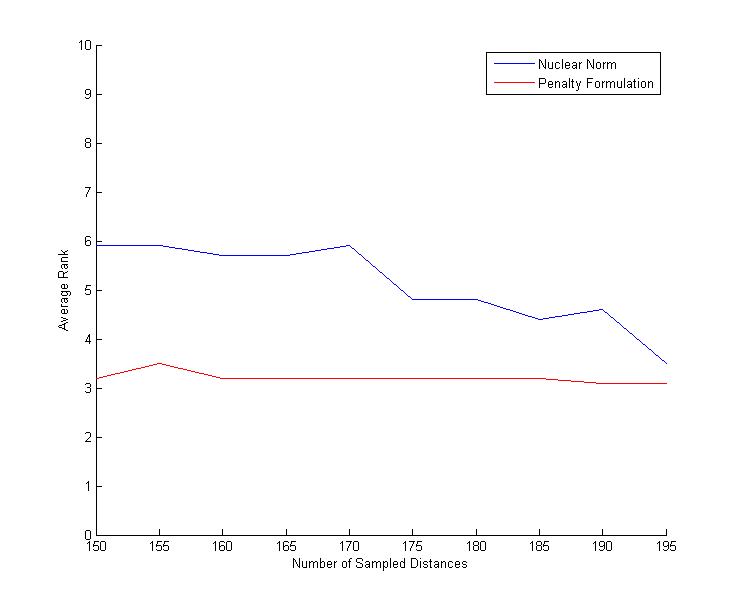}
     \end{center}
    \caption[Rank of Solutions When more Distances are Sampled]{\small  Average rank of solutions for the 20 instances,
   as the number of sampled distances increases from 150 to 195.}
    \label{Rank150to195}
\end{figure}
Figure \ref{Rank150} shows that when 150 distances are sampled, the solutions returned
by the penalty formulation have notably lower rank when compared with the convex approximation.
There was only one instance where the penalty formulation failed to find a solution
with rank~3; in contrast, the lowest rank solution found by the nuclear norm approach
had rank 5, and that was only for one instance.

We also experimented with various numbers of sampling distances from 150  to 195.
For each number, we randomly sampled that number of distances,
then compared the average rank returned by the penalty formulation and the convex approximation.
Figure \ref{Rank150to195} shows that the penalty formulation is more likely to recover a low rank matrix when the number of sampled distances is the same.
The nuclear norm approach appears to need about 200 sampled distances in order to obtain a
solution with rank 3 in most cases.
There has been some research on maximizing the nuclear norm os symmetric matrices to approximately
minimize rank~\cite{krislock3}. However, for our (very sparse) instances we obtained similar ranks either minimizing
or maximizing the nuclear norm.

For the 20 cases where 150 distances are sampled, the average time for CVX is 0.6590 seconds,
while for the penalty formulation the average time for the fast PALM method is 10.21 seconds.
Although the fast PALM method cannot beat CVX in terms of speed,
it can solve the problem in a reasonable amount of time and produces a lower rank solution
for our test instances.

\section{Conclusions}

The SDCMPCC approach gives an equivalent nonlinear programming
formulation for the combinatorial optimization problem of minimizing the rank
of a symmetric positive semidefinite matrix subject to convex constraints.
The disadvantage of the SDCMPCC approach lies in the nonconvex
complementarity constraint, a type of constraint for which constraint
qualifications might not hold.
We showed that the calmness constraint qualification holds for the SDCMPCC formulation
of the rank minimization problem, provided the convex constraints satisfy
an appropriate condition.
We developed a penalty formulation for the problem which satisfies calmness
under the same condition on the convex constraint.
The penalty formulation could be solved effectively using
an alternating direction approach, accelerated through the use of a
momentum term.
For our test problems, our formulation outperformed a nuclear norm approach,
in that it was able to recover a low rank matrix using fewer samples
than the nuclear norm approach.

There are alternative nonlinear approaches to rank minimization problems,
and it would be interesting to explore the relationships between the methods.
The formulation we've presented is to minimize the rank;
the approach could be extended to handle problems with
upper bounds on the rank of the matrix, through the use of a constraint
on the trace of the additional variable~$U$.
Also of interest would be the extension of the approach to the nonsymmetric case.

\bibliographystyle{siamplain}

\bibliography{optim}

\end{document}